\documentclass[10pt]{amsart}
\title
[Algebraic hulls and exponential iterated  integrals  ]
{Algebraic hulls of solvable groups and exponential iterated  integrals on solvmanifolds }

\author{Hisashi Kasuya}
\usepackage{amssymb}
\usepackage{amsmath}
\usepackage{amscd}
\usepackage{amstext}
\usepackage{amsfonts}
\usepackage[all]{xy}
\usepackage{eucal}

\address[H.kasuya]{Graduate school of mathematical science university of tokyo japan }
\curraddr{}
\email{khsc@ms.u-tokyo.ac.jp}

\thanks{}
\keywords{exponential iterated integral, algebraic hull, solvmanifold}

\newcommand{\C}{\mathbb{C}}
\newcommand{\R}{\mathbb{R}}

\newcommand{\Z}{\mathbb{Z}}
\newcommand{\g}{\frak{g}}
\newcommand{\n}{\frak{n}}

\newcommand{\Aut}{\rm Aut}

\theoremstyle{plain}
\newtheorem{theorem}{Theorem}[section]
\theoremstyle{plain}
\newtheorem{remark}{Remark}[section]
\theoremstyle{lemma}
\newtheorem{lemma}[theorem]{Lemma}
\theoremstyle{definition}

\theoremstyle{proposition}
\newtheorem{proposition}[theorem]{Proposition}
\theoremstyle{corollary}
\newtheorem{corollary}[theorem]{Corollary}
\theoremstyle{remark}

\begin{document} 
\maketitle
\begin{abstract}
We represent the coordinate ring of algebraic hulls   (which are generalizations of the Malcev completions of nilpotent  groups for solvable  groups)   of solvmanifolds $G/\Gamma$ by using  Miller's exponential  iterated integrals (which are extensions of Chen's iterated integrals) of invariant differential  forms.
\end{abstract}
\section{\bf Introduction}
Let $G$ be a simply connected  Lie group 
and $\g$ the Lie algebra of $G$.
We consider the space  $\bigwedge \g_{\C}^{\ast}$   of $\C$-valued left $G$-invariant differential forms on $G$.
We assume that $G$ has a lattice (i. e. cocompact discrete subgroup) $\Gamma$.
We consider the compact homogeneous space $G/\Gamma$ and $\bigwedge \g_{\C}^{\ast}$ as a subcomplex of the de Rham complex $A^{\ast}_{\C}(G/\Gamma)$ of $G/\Gamma$.
Suppose $G$ is nilpotent.
Then  we have the unique unipotent algebraic group ${\bf U}_{\Gamma}$ called the Malcev completion of $\Gamma$ such that there is a injection $\Gamma\to {\bf U}_{\Gamma}$ with the Zariski-dense image.
We can represent the coordinate ring of ${\bf U}_{\Gamma}$ by using  Chen's iterated integrals on $G/\Gamma$  (see \cite{CH}).
Since the inclusion $\bigwedge \g_{\C}^{\ast}\subset A^{\ast}_{\C}(G/\Gamma)$ induces a cohomology isomorphism by Nomizu's theorem \cite{Nom}, $\bigwedge \g_{\C}^{\ast}$ is the Sullivan minimal model of $ A^{\ast}_{\C}(G/\Gamma)$ (see \cite{H}).
This implies $H^{0}(\bar B (\bigwedge \g_{\C}^{\ast}))\cong H^{0}(\bar B( A^{\ast}_{\C}(G/\Gamma)))$ where $\bar B (\bigwedge \g_{\C}^{\ast})$ and $\bar B( A^{\ast}_{\C}(G/\Gamma))$ are the reduced bar constructions of $\bigwedge \g_{\C}^{\ast}$ and  $A^{\ast}_{\C}(G/\Gamma)$ respectively (see \cite{CH2}).
Hence we can represent the coordinate ring of ${\bf U}_{\Gamma}$ by using  Chen's iterated integrals of left-invariant forms.

Suppose $G$ is solvable.
Then Chen's iterated integrals  on  $G/\Gamma$ does not give sufficient information on the fundamental group of $G/\Gamma$.
For example, let $G=\R\ltimes_{\phi} \R^{2}$ such that $\phi(t)=\left(
\begin{array}{cc}
e^{ t}&0\\
0&e^{-t}
\end{array}
\right) $.
Then $G$ has a lattice $\Gamma$ and the inclusion $\bigwedge \g^{\ast}_{\C}\subset A_{\C}^{\ast}
(G/\Gamma)$ induces a cohomology isomorphism (see \cite{Hatt}).
Since we have $H^{1}(G/\Gamma,\C)=H^{1}(\bigwedge \g^{\ast}_{\C})=\C$, by Chen's results,
iterated integrals represent the coordinate ring of a additive algebraic group  ${\mathbb G}_{\rm ad}=\C$ (see \cite{M}).
But since $\Gamma$ is solvable and not abelian, we can't embed $\Gamma$ in ${\mathbb G}_{\rm ad}$.

 In \cite{Mos2}, as an extension of the Malcev completion, Mostow constructed a Zariski-dense embedding of $\Gamma$
in an algebraic group  ${\bf H}_{\Gamma}$ called algebraic hull of $\Gamma$.
In \cite{M}, Miller gave  extensions of Chen's iterated integrals called  exponential iterated integrals.
In this paper we represent the coordinate ring of ${\bf H}_{\Gamma}$ by using  Miller's exponential iterated integrals of left-invariant differential forms on $G/\Gamma$.

\section{\bf Relative completions and algebraic hulls}

Let $G$ be a discrete group (resp. a Lie group).
We call a map $\rho:G\to GL_{n}(\C)$ a representation, if $\rho$ is a homomorphism  of  groups (resp. Lie groups).
In this paper   we denote by $T_{n}(\C)$ the group of $n\times n$ upper triangular matrix and denote by $U_{n}(\C)$ the group of $n\times n$  upper triangular unipotent matrix.

\subsection{\bf algebraic groups and pro-algebraic groups}
In this paper an algebraic group means an affine algebraic variety $\bf G$ over $\C$ with a group structure such that the multiplication and inverse are morphisms of varieties.
All algebraic groups arise as Zariski-closed subgroups of $GL_{n}(\C)$.
A pro-algebraic group is an inverse limit of algebraic groups.
If a pro-algebraic group is an inverse limit of unipotent algebraic groups, it is called pro-unipotent.
Let $\bf G$ be a pro-algebraic group.
We denote by ${\bf U}({\bf G})$ the maximal pro-unipotent normal subgroup called the pro-unipotent radical.
If ${\bf U}({\bf G})=e$, $\bf G$ is called reductive.
We denote by $\C[\bf G]$ the coordinate ring of $\bf G$.
The group structure on $\bf G$ induces  a Hopf algebra structure on $\C[\bf G]$.
It is known that we have the anti-equivalence between algebras and affine varieties induces an anti-equivalence between pro-algebraic groups  and  reduced Hopf algebras.

\begin{theorem}{\rm (\cite{Mos1},\cite{HM})} \label{sppp}
Let $\bf G$ be a pro-algebraic group. Then the exact sequence 
\[\xymatrix{
	1\ar[r]&{\bf U}(\bf G) \ar[r]& {\bf G} \ar[r] & {\bf G}/{\bf U}(\bf G) \ar[r] &1 
	 }
\]
splits.
\end{theorem}
Let $G$ be a discrete group or Lie group.
We denote by $A(G)$ the inverse limit of all representations $\phi :G\to \bf G$ with Zariski-dense images.
We call the pro-unipotent radical ${\bf U}(A(G))$ of $A(G)$ the unipotent hull of $G$ and denote it ${\bf U}_{G}$.

\subsection{\bf Relative completion}
Let $\rho:G\to {\bf S}$ be a representation of $G$ to a diagonal algebraic group ${\bf S}$ with the Zariski-dense image.
Let $\phi: G\to {\bf G}$ be a representation of  $G$ to an algebraic group $\bf G$ with the Zariski-dense image.
We call $\phi$ a $\rho$-relative representation if we have the commutative diagram 
\[\xymatrix{
	1\ar[r]&{\bf U}(\bf G) \ar[r]& {\bf G} \ar[r] &{\bf S} \ar[r] &1 \\
	& &G\ar^{\phi}[u]	 \ar^{\rho}[ru]
 }.
\]

If $\bf S$ is contained in an algebraic torus, for any $\rho$-relative 
representation $\phi:G\to {\bf G}$ there exists a faithful 
representation ${\bf G}\hookrightarrow T_{n}(\C)$ such that ${\bf G}\cap 
U_{n}(\C)={\bf U}({\bf G})$(see \cite{M}).

We denote by $\mathcal G_{\rho}(G)$ the inverse limit of   $\rho$-relative representations $\phi_{i}:G\to {\bf G}_{i}$.
We call $\mathcal G_{\rho}(G)$ the $\rho$-relative completion of $G$.
If $\bf S$ is trivial, $\mathcal G_{\rho}(G)$ is the classical Malcev (or unipotent) completion.

\subsection{\bf Algebraic hulls}
We define the algebraic hulls of polycyclic groups (resp. Lie groups) which constructed in {\cite{Mos2}.

A group $\Gamma$ is polycyclic if it admits a sequence 
\[\Gamma=\Gamma_{0}\supset \Gamma_{1}\supset \cdot \cdot \cdot \supset \Gamma_{k}=\{ e \}\]
of subgroups such that each $\Gamma_{i}$ is normal in $\Gamma_{i-1}$ and $\Gamma_{i-1}/\Gamma_{i}$ is cyclic.
For a polycyclic group $\Gamma$, we denote ${\rm rank}\,\Gamma=\sum_{i=1}^{i=k} {\rm rank}\,\Gamma_{i-1}/\Gamma_{i}$.
Let $G$ be a simply connected solvable Lie group and $\Gamma$ be a lattice of $G$.
Then $\Gamma$ is torsion-free polycyclic and $ \dim G={\rm rank}\,\Gamma$.

Let $G$ be a simply connected solvable Lie group or torsion-free polycyclic group.
Consider the algebraic completion $A(G)$.
Then it is known that the unipotent hull ${\bf U}_{G}={\bf U}(A(G))$ is finite dimensional (see \cite{Mos2}).
By Theorem \ref{sppp}, we have a splitting $A(G)=(A(G)/{\bf U}_{G})\ltimes_{\phi} {\bf U}_{G}$.
Let $K$ be the kernel of the action $\phi:(A(G)/{\bf U}_{G})\to {\rm Aut}({\bf U}_{G})$.
Then $K$ is the maximal reductive normal subgroup of $A(G)$ (see \cite{Mos2}).
Denote ${\bf H}_{G}=A(G)/K$ and call it the algebraic hull of $G$.

\begin{theorem}{\rm (\cite{Mos2},\cite{R})} \label{AL}
Let $G$ be a simply connected solvable Lie group (resp. torsion-free 
polycyclic group).
Then $G \to {\bf H}_{G}$ is injective and ${\bf H}_{G}$ is a finite dimensional algebraic group such that:\\
$(1)$ $\dim {\bf U}({\bf H}_{G})=\dim G$ (resp. $={\rm rank}\, G $).\\
$(2)$The centralizer of ${\bf U}({\bf H}_{G})$ in ${\bf H}_{G}$ is 
contained in ${\bf U}({\bf H}_{G})$.\\
Conversely if an algebraic group $\bf H$ with an injective homomorphism $\
\psi:G\to {\bf H}$ with the Zariski-dense image satisfies the properties $(1)$ and $(2)$,
 then ${\bf H}$ is isomorphic to ${\bf H}_{G}$.
\end{theorem}

\subsection{Direct constructions of algebraic hulls}
The idea of this subsection is based on classical works of semi-simple splitting (see \cite{Re},  \cite{OV} and the references given there).
Let $\g$ be a solvable Lie algebra, and $\n=\{X\in \g\vert {\rm ad}_{X}\, \, {\rm is\,\, nilpotent}\}$.
$\n$ is the maximal nilpotent ideal of $\g$ and called the nilradical of $\g$.
 Then we have $[\g, \g]\subset \n$.
 Let $D(\g)$ be the Lie algebra of derivations of $\g$.
 By the Jordan decomposition, we have  ${\rm ad}_{X}={\rm ad}_{sX}+{\rm ad}_{nX}$ such that ${\rm ad}_{sX} $ is a semi-simple operator and ${\rm ad}_{nX}$ is  a nilpotent operator.
 Since we have $d_{X}$, $n_{X}$ $\in D(\g)$,
  we have the map ${\rm ad}_{s}:\g\to D(\g)$.
Since $\rm ad$ is trigonalizable  (Lie's theorem), this map is homomorphism with the kernel $\n$.
 Let $\bar{\g} ={\rm Im} \,{\rm ad}_{s}\ltimes\g$.
and $\bar{\n}=\{X-{\rm ad}_{sX}\in  \bar{\g}  \vert X\in\g\}$.
\begin{proposition}\label{spli}
$\bar{\n}$ is a nilpotent ideal of $\bar{\g}$ and we have a decomposition $\bar{\g}={\rm Im}\, {\rm ad}_{s} \ltimes\bar{\n}$.
\end{proposition}
\begin{proof}
 By  ${\rm ad}_{X-{\rm ad}_{sX}}={\rm ad}_{X}-{\rm ad}_{sX}$ on $\g$,  ${\rm ad}_{X-{\rm ad}_{sX}}$ is a nilpotent operator and
hence $\bar{\n}$ consists of nilpotent elements. 
 By  Lie's theorem, we have a basis 
\[X_{1},\dots,X_{l},X_{l+1}\dots ,X_{n}\]
 of $\g \otimes \C$ such that ${\rm ad}$
  is represented  by upper triangular matrices.
Since the nilradical $\n$ is an ideal, $\n\otimes \C$ is ${\rm ad}$-invariant subspace of $\g \otimes \C$.
We choose $X_{1},\dots,X_{l}$ a basis of  $\n\otimes \C$.
By $[\g, \g]\subset \n$, we have ${\rm ad}_{X}(\g\otimes \C)\subset \n\otimes \C=\langle X_{1},\dots ,X_{l}\rangle$, and hence ${\rm ad}$ represented as
\[{\rm ad}_{X}=\left(
\begin{array}{cccccc}
a_{11}(X)&\dots&&&\dots&a_{1l}(X)\\
&\ddots&&&&\vdots\\
&&a_{ll}(X)&\dots &&a_{lm}(X)\\
&&&0&\dots&0\\
&&&&\ddots&\vdots\\
&&&&&0
\end{array}
\right).
\]
Thus we have 
${\rm ad}_{sX}(X_{i})=a_{11}(X)X_{i} $ for $ 1\le i \le l$ and
$ {\rm ad}_{sX}(X_{i})=0$ for $l+1\le i \le n$.
By this we have 
\[
[X_{i}+{\rm ad}_{sY},X_{j}+{\rm ad}_{sZ}]\in\langle X_{1},\dots ,X_{l}\rangle=\n\otimes \C\]
for any $1\le i,j\le n$, $Y,Z\in \g$.
This implies $[\bar{\g}, \bar{\g}]\subset\n$.
 By $\n\subset \bar\n$,   $ \bar{\n}$ is an ideal of $\bar{\g}$ and
we have  $\bar{\g}=\{{\rm ad}_{sX}+Y-{\rm ad}_{sY}\vert X, Y\in \g\}={\rm Im}\, {\rm ad}_{s}\ltimes\bar{\n}$.
\end{proof}
By this proposition,  we have the inclusion $i:\g\to D(\bar{\n})\ltimes \bar{\n}$ given by $i(X)={\rm ad}_{sX}+X-{\rm ad}_{sX}$ for $X\in\g$.

Let $G$ be a simply connected solvable Lie group and $\g$ be the Lie algebra of $G$.
For the adjoint representation ${\rm Ad}:G\to {\rm Aut}(\g)$,
we consider  the semi-simple part ${\rm Ad}_{s}: G\to   {\rm Aut}(\g)$ as similar to the Lie algebra case.
Denote by $T$ the universal covering of ${\rm Ad}_{s}( G)$.
Let $\bar{N}$ be the simply connected Lie group which corresponds to $\bar{\n}$.
Then by Proposition \ref{spli}, we have $T\ltimes G=T\ltimes \bar N$.
By the proof of this proposition, the action $T\to{\rm Aut} (\bar \n)$ is the extension of the action of ${\rm Im}\, {\rm ad}_{s} $.
Hence we have ${\rm Ad}_{s}(G)\ltimes G={\rm Ad}_{s}(G)\ltimes \bar N$.

A simply connected nilpotent Lie group is considered as the real points of a unipotent $\R$-algebraic group (see \cite[p. 43]{OV}) by the exponential map.
We have the unipotent $\R$-algebraic group $\bf \bar{N}$ with ${\bf \bar{N}}(\R)=\bar{N}$.
We identify $\Aut_{a}({\bf \bar{N}})$ with $\Aut(\n_{\C})$ and $\Aut_{a}({\bf \bar{N}})$ has the $\R$-algebraic group structure with ${\rm Aut}_{a}({\bf \bar{N}})(\R)= {\rm Aut} (N)$.
So we have the $\R$-algebraic group $ \Aut_{a} (\bf\bar{N})\ltimes \bar{N}$.
Then by ${\rm Ad}_{s}(G)\ltimes G={\rm Ad}_{s}(G)\ltimes \bar N\subset  \Aut_{a} (\bf\bar{N})\ltimes \bar{N}$, we consider the Zariski-closure $\bf G$ of $G$ in $ \Aut_{a} (\bf\bar{N})\ltimes \bar{N}$.
Since ${\rm Ad}_{s}(G)$ is a group of diagonal automorphisms, we have ${\bf U}({\bf G})= \bar{\bf N}$.
By $\dim G=\dim \bar N$, we can easily check that $\bf G$  satisfies the properties (1), (2)  in Theorem \ref{AL} and hence it is the algebraic hull ${\bf H}_{G}$ of $G$.
Hence  the inclusion $i:\g\to D(\bar{\n})\ltimes \bar{\n}$ induces the algebraic hull $I:G\to {\bf H}_{G}$ of $G$.
Since $i:\g\to D(\bar{\n})\ltimes \bar{\n}$ is given by $i(X)={\rm ad}_{sX}+X-{\rm ad}_{sX}\in D(\bar{\n})\ltimes \bar{\n}$ ,   the composition $G\to {\bf H}_{G}\to {\bf H}_{G}/{\bf U}({\bf H}_{G})$ is induced by the Lie algebra homomorphism  ${\rm ad}_{s}:\g\to D(\g)$ by ${\bf U}({\bf G})= \bar{\bf N}$. 
Thus we have the following lemma.
\begin{lemma}
 The algebraic hull $ G\to {\bf H}_{G}$ is an ${\rm Ad}_{s}$-relative representation.
\end{lemma}

\subsection{\bf Algebraic hulls and relative completions of solvable groups}

\begin{theorem}\label{relal}
Let $G$ be a simply connected Lie group.
Then the algebraic hull $ {\bf H}_{G}$ is isomorphic to the ${\rm Ad}_{s}$-relative completion  ${\mathcal G}_{{\rm Ad}_{s}}(G)$ of $G$.
\end{theorem}
\begin{proof}
Consider a commutative diagram
\[\xymatrix{
	{\bf H}^{\prime} \ar[r]^{\Phi}&{\bf H}_{G}   \\
	G\ar[u]	 \ar[ru]
 }
\] 
for some ${\rm Ad}_{s}$-relative representation $G\to {\bf H}^{\prime}$.
Since $G\to {\bf H}^{\prime}$ and $G\to {\bf H}_{G}$ have Zariski-dense images, $\Phi : {\bf H}^{\prime} \to {\bf H}_{G}$ is surjective and the restriction $\Phi :{\bf U}({\bf H}^{\prime}) \to {\bf U} ({\bf H}_{G})$ is also surjective.
By ${\bf U} ({\bf H}_{G})={\bf U}_{G}$, $\Phi :{\bf U}({\bf H}^{\prime }) \to {\bf U} ({\bf H}_{G})$ is an isomorphism. 
Since $G\to {\bf H}^{\prime}$ and $G\to {\bf H}_{G}$ are ${\rm Ad}_{s}$-relative representations,  $\Phi$ induces the isomorphism ${\bf H}^{\prime}/ {\bf U} ({\bf H}^{\prime}) \to {\bf H}_{G}/{\bf U}({\bf H}_{G})$.
Hence $\Phi : {\bf H}^{\prime}\to {\bf H}_{G}$ is an isomorphism.
By the definition of ${\rm Ad}_{s}$-relative completion of $G$, we have the theorem. 
\end{proof}

\begin{theorem}\label{alrr}
Let $G$ be a simply connected solvable Lie group and $\Gamma$ a lattice of $G$.
Then the algebraic hull ${\bf H}_{\Gamma}$ of  $\Gamma$ is isomorphic to ${\rm Ad}_{s\vert_{\Gamma}}$-relative completion ${\mathcal G}_{{\rm Ad}_{s\vert_{\Gamma}}}(\Gamma)$ of $\Gamma$. 
\end{theorem}
\begin{proof}
For the algebraic hull $\psi :G\to {\bf H}_{G}$ of $G$, we consider the Zariski-closure of $\psi(\Gamma)$ in ${\bf H}_{G}$.
Then by $\dim G={\rm rank}\, \Gamma$ we can easily check that this algebraic group satisfies (1) and (2) in Theorem \ref{AL} and hence it is the algebraic hull ${\bf H}_{\Gamma}$ of $\Gamma$.
By the above theorem, $\Gamma\to {\bf H}_{\Gamma}$ is a ${\rm Ad}_{s\vert_{\Gamma}}$-relative representation.
As similar to the above proof,  we have the theorem.
\end{proof}

\section{Exponential iterated integral on solvmanifolds}
In this section we consider Miller's exponential iterated integrals.
Let $M$ be a $C^{\infty}$-manifold and $\Omega_{x}M$ be a space of piecewise smooth loops
$\lambda :[0,1]\rightarrow M$ with $\lambda(0)=x$.
For $1$-forms $\omega_{1},\dots ,\omega_{n}\in A^{\ast}_{\C}(M)$, the iterated integral  $\int \omega_{1} \omega_{2} \cdot \cdot \cdot \omega_{n}:\Omega_{x}M\to \C$ is defined by
\[\int_{\lambda} \omega_{1} \omega_{2} \cdot \cdot \cdot \omega_{n} = \int_{0\le t_{1}\le t_{2} \le \cdot \cdot \le t_{n} \le 1}F(t_{1})F(t_{2})\cdot \cdot \cdot F(t_{n}) dt_{1} dt_{2}\cdot \cdot \cdot dt_{n} \\
\lambda \in \Omega_{x}M
\]
where $F_{i}(t)dt=\lambda ^{\ast} \omega_{i} \in A^{1}([0,1])$.
In \cite{M}, for $\delta_{1}, \delta_{2}, \cdot \cdot \cdot ,\delta_{n}, \omega_{12}, \omega_{23} ,\cdot \cdot \cdot, \omega_{n-1n} \in A^{1}_{\C}(M)$ Miller defined the exponential iterated integral $\int e^{\delta_{1}}\omega_{12}e^{\delta_{2}}\omega_{23}\cdot\cdot \cdot e^{\delta_{n-1}}\omega_{n-1n}e^{\delta_{n}} :\Omega_{x}M\to \C$ as 
\[\int_{\lambda} e^{\delta_{1}}\omega_{12}e^{\delta_{2}}\omega_{23}\cdot\cdot \cdot e^{\delta_{n-1}}\omega_{n-1n}e^{\delta_{n}}\]
\[=
\sum_{m_{1},m_{2},\cdot \cdot \cdot m_{n} \ge 0}\int_{\lambda}\underbrace{\delta_{1}\dots \delta_{1}}_{m_{1} \ terms}\omega_{12}\underbrace{\delta_{2}\dots \delta_{2}}_{m_{2} \ terms}\cdot \cdot \cdot \omega_{n-1n}\underbrace{\delta_{n}\dots \delta_{n}}_{m_{n} \ terms}.
\]
 Then this infinite sum converges (see \cite{M}).
Let $L\subset A^{1}_{\C}(M)$ be a finitely generated $\Z$-module of  $1$-forms such that $dL=0$.
We denote $E^{L}(M,x)$  the $\C$-vector space of functions on $\Omega_{x}M$ generated by 
\[\{ \int e^{\delta_{1}}\omega_{12}\cdot \cdot \cdot \omega_{n-1n}e^{\delta_{n}}\vert \delta_{1},\cdot \cdot \cdot \delta_{n} \in L ,
 \ \omega_{12}, \omega_{23} ,\cdot \cdot \cdot, \omega_{n-1n} \in A^{\ast}_{\C}(M)\}.\]
If $I \in E^{L}(M,x)$ is constant on homotopy classes of loops $\lambda:[0,1]\to M$ relative to $\{0, 1\}$, we call $I$ a closed exponential iterated integral.
Let $H^{0}(E^{L}(M,x))$ denote the subspace of closed exponential iterated integrals. 
Take a $\Z$-basis $\{\delta_{1}, \delta_{2}, \dots ,\delta_{n}\}$ of $L$.
Then we have the diagonal representation $\rho:\pi_{1}(M,x)\to D_{n}(\C)$ such that 
$
\rho(\lambda)={\rm diag} (\int_{\lambda}e^{\delta_{1}},\dots ,\int_{\lambda}e^{\delta_{n}})
$
for $\lambda\in\pi_{1}(M,x)$. 
Consider the $\rho$-relative completion ${\mathcal G}_{\rho}(\pi_{1}(M,x))$ of the fundamental group of $M$.
Miller showed the following theorem.
\begin{theorem}{\rm (\cite[Theorem 6.1]{M})}
The space $H^{0}(E^{L}(M,x))$ is a Hopf algebra and  we have a Hopf algebra isomorphism
\[H^{0}(E^{L}(M,x))\cong \C[{\mathcal G}_{\rho}(\pi_{1}(M,x))].
\]
\end{theorem}
\begin{remark}\label{SAII}
For any $\rho$-relative representation $\phi:\pi_{1}(M,x)\to {\bf G}$, Miller showed that $\phi$ is the monodromy of a flat connection $\omega$  on $M\times \C^{n}$ whose connection form is an upper triangular matrix.
Then the monodromy of $\omega$ is given by $ I+\sum^{\infty}_{i=1} \int \underbrace{\omega \omega \cdot \cdot \cdot \omega}_{i \ terms}$ and its  matrix entries are exponential iterated integrals.
In the proof of Theorem 6.1 of \cite{M}, Miller showed that these matrix entries generate the coordinate ring $\C[{\bf G}]$.
\end{remark}
Consider a simply connected solvable Lie group $G$ with a lattice $\Gamma$.
Take a diagonalization of the semi-simple part ${\rm ad}_{s}$ of the adjoint representation ${\rm ad}$ on $\g$. 
Write
${\rm ad}_{s}={\rm diag}(\delta_{1},\dots ,\delta_{n})
$
where $\delta_{1},\dots ,\delta_{n}$ are characters of $\g$.
By $\delta_{1},\dots ,\delta_{n}\in {\rm Hom}(\g,\C)$,
we regard  $\delta_{1},\dots ,\delta_{n}$ as left-invariant closed $1$-forms.
Let $L$ be  the $\Z$-module generated by $\delta_{1},\dots ,\delta_{n}$.
Consider the algebraic hull ${\bf H}_{\Gamma}$ of $\Gamma$.
Since we have $\pi_{1}(G/\Gamma,x)\cong\Gamma$, by Theorem \ref{alrr}, we have:
\begin{corollary}\label{ALGH}
We have a Hopf algebra isomorphism
\[
H^{0}(E^{L}(G/\Gamma,x))\cong  \C[{\bf H}_{\Gamma}].
\]
\end{corollary}

Let $E^{L}( \g^{\ast}_{\C})$ denote the subvector space of $E^{L}(G/\Gamma,x)$ generated by 
\[\{ \int e^{\delta_{1}}\omega_{12}\cdot \cdot \cdot \omega_{n-1n}e^{\delta_{n}}\vert \delta_{1},\cdot \cdot \cdot \delta_{n} \in L 
\ \ \ \omega_{12}, \omega_{23} ,\cdot \cdot \cdot, \omega_{n-1n} \in  \g^{\ast}_{\C}\}.\]
Studying the proof of \cite[Lemma 5.1]{M}, we can see that $E^{L}( \g^{\ast}_{\C})$ is closed under the multiplication.
We define the subring 
\[H^{0}(E^{L}(\g^{\ast}_{\C}))=E^{L}(\g^{\ast}_{\C})\cap H^{0}(E^{L}(G/\Gamma,x))\]
 of $H^{0}(E^{L}(G/\Gamma,x))$.

\begin{theorem}\label{INVT}
We have $H^{0}(E^{L}(\g^{\ast}_{\C}))= H^{0}(E^{L}(G/\Gamma,x))$.
\end{theorem}
\begin{proof}
Consider the algebraic hull $\psi:G\to {\bf H}_{G}$ of $G$.
Since $\psi:G\to {\bf H}_{G}$  is ${\rm Ad}_{s}$-relative, we can assume ${\bf H}_{G}\subset T_{r}(\C)$ and $U_{r}(\C)\cap {\bf H}_{G}= {\bf U}({\bf H}_{G})$ as in Section 2.2.
Let $\psi_{\ast}:\g\to {\frak t}_{r}(\C)$ be the derivative of $\psi$ where $ {\frak t}_{r}(\C)$ is the Lie algebra of $T_{r}(\C)$.
We write
\[\psi_{\ast}=\left(
\begin{array}{ccccc}
\omega_{11}&\omega_{12}&\cdots&&\omega_{1r}\\
&\ddots&\ddots&&\vdots\\
&&\ddots& &\\
&&&\ddots&\omega_{r-1r}\\
&&&&\omega_{rr}

\end{array}
\right) 
\]
  as we consider $\psi_{\ast}\in {\rm Hom}(\g,\C)\otimes T_{r}(\C)$.
Then we have 
\[(d\psi_{\ast}-\psi_{\ast}\wedge\psi_{\ast})(X,Y)=\psi_{\ast}([X,Y])-[\psi_{\ast}(X),\psi_{\ast}(Y)]=0
\]
for $X,Y\in \g$.
Hence we have the flat connection $d-\psi_{\ast}$ on the vector bundle $G\times \C^{r}$.
Consider the  parallel transport $T=I+\sum^{\infty}_{i=1} \int \psi_{\ast}  \cdot \cdot \cdot \psi_{\ast}$ of this connection.
Let $P_{e}G$ be the space of the paths $\gamma:[0,1]\to G$ with $\gamma(0)=e$ where $e$ is the identity element of $G$.
We consider the spaces $P_{e}G/\sim$ of homotopy classes of  $\gamma \in P_{e}G$ relative to $\{0,1\}$.
Since $G$ is simply connected, we have $P_{e}G/\sim=G$.
It is easily seen that the  parallel transport $T$  on  $P_{e}G/\sim=G$ is a homomorphism whose derivative is equal to $\psi_{\ast}$.
Hence  we can identify  the  parallel transport $T$  on  $P_{e}G/\sim$ with the representation $\psi$.
Since $\psi$ is ${\rm Ad}_{s}$-relative and the diagonal entries of $T$ are $\int e^{\omega_{11}},\dots ,\int e^{\omega_{rr}}$, we have $\omega_{11},\dots ,\omega_{rr}\in L$.
 By the proof of Theorem \ref{alrr}, the injection $\phi:\Gamma\to {\bf H}_{\Gamma}$ is the restriction of $\psi$ on $\Gamma$.
Thus the representation $\phi$ is the monodromy $I+\sum^{\infty}_{i=1} \int \psi_{\ast}  \cdot \cdot \cdot \psi_{\ast}$ of the left-invariant flat connection 
$d-\psi_{\ast}$ on the vector bundle $G/\Gamma\times \C^{r}$.
By Remark \ref{SAII}, the ring $\C[{\bf H}_{\Gamma}]$ is generated by matrix entries of $I+\sum^{\infty}_{i=1} \int \psi_{\ast}  \cdot \cdot \cdot \psi_{\ast}$.
Hence the theorem follows from Corollary \ref{ALGH}.
\end{proof}

\section{\bf An  Example and a remark}
Let $N$ be a simply connected nilpotent Lie group and $\n$ the Lie algebra of $N$.
We suppose that $G$ has a lattice $\Gamma$.
Then we can represent the coordinate ring of the Malcev completion of $\Gamma$ by using  Chen's iterated integral of left-invariant forms on $N$.
In this paper we give another representation of the Malcev completion of the fundamental group of some nilmanifold.

Consider the solvable Lie group $G=\R\ltimes_{\mu} \C^{2}$ such that
$\mu (t)=\left(
\begin{array}{cc}
e^{i\pi t}&te^{i\pi t}\\
0&e^{i \pi t}
\end{array}
\right)
$.
We have the  lattice  $\Gamma
=2\Z\ltimes (\Z+i\Z)$.
We consider the inclusion $\bigwedge \g^{\ast}_{\C} \subset A^{\ast}(G/\Gamma)$.
The map $H^{\ast}( \bigwedge \g^{\ast}_{\C})\to H^{\ast}(G/\Gamma, \C )$ induced by this inclusion is injective (see \cite{R}).
By $(\bigwedge \g^{\ast}_{\C})^{0}=\C$ and $(\bigwedge \g^{\ast}_{\C})^
{1}\cap dA^{0}(G/\Gamma)=0$, we have an isomorphism
\[ H^{0}(B(\bigwedge \g^{\ast}_{\C}, x))\cong H^{0}
(\bar{B}(\bigwedge \g^{\ast}_{\C}))\]
 where $H^{0}(B(\bigwedge \g^{\ast}_{\C},x) )$ is the space of closed Chen's iterated integrals of the left-invariant forms on the based loop space $\Omega_{x}G/\Gamma$ and  $H^{0}(\bar{B}(\bigwedge \g^{\ast}_{\C}))$ is the 
reduced bar construction (see \cite{CH}).
Since we have $H^{1}(\bigwedge \g^{\ast}_{\C})\cong \C$,
 we have an isomorphism $ H^{0}(B(\bigwedge \g^{\ast}_{\C},x) )\cong \C[{\mathbb G}_{\rm ad}]$.
 
On the other hand, let $L$ be the sub $\Z$-module of  $\g^{\ast}_{\C}$ generated 
by 
$\{ i\pi dt \}$.
Then by Corollary \ref{ALGH} and  Theorem \ref{INVT}, we have an isomorphism
\[H^{0}(E^{L}(\g^{\ast}_{\C}))\cong \C[{\bf H}_{\Gamma}].\]
Since we have $\mu (2t)=\left(
\begin{array}{cc}
1&2t\\
0&1
\end{array}
\right)$
for $t\in\Z$, $\Gamma$ is nilpotent.
Hence ${\bf H}_{\Gamma}$ is the Malcev completion of $\Gamma$.
Since two compact solvmanifolds having the same fundamental group are 
diffeomorphic (see \cite{Mosr} or \cite{R}),
 $G/\Gamma$ is diffeomorphic to a nilmanifold.
By these arguments we notice:
\begin{remark}
By  closed Chen's iterated integrals of the $1$-forms $\g_{\C}^{\ast}$ on $G/\Gamma$, we can not represent the coordinate ring of Malcev completion of the fundamental 
group of the nilmanifold $G/\Gamma$.
But  the closed $L$-exponential iterated integrals of $\g_{\C}^{\ast}$ enable 
us to represent it. 
\end{remark}

{\bf  Acknowledgements.} 

The author would like to express his gratitude to   Toshitake Kohno for helpful suggestions and stimulating discussions.
This research is supported by JSPS Research Fellowships for Young Scientists.

\end{document}